\documentclass[11pt]{article}
\usepackage[a4paper,margin=1in]{geometry}
\usepackage[T1]{fontenc}
\usepackage[utf8]{inputenc}
\usepackage{lmodern,microtype}
\usepackage{amsmath,amssymb,amsthm,mathtools,bm,booktabs,array,enumitem}
\usepackage[hidelinks]{hyperref}
\usepackage[nameinlink,capitalize]{cleveref}
\newtheorem{theorem}{Theorem}[section]
\newtheorem{proposition}[theorem]{Proposition}
\newtheorem{lemma}[theorem]{Lemma}
\newtheorem{corollary}[theorem]{Corollary}
\theoremstyle{definition}
\theoremstyle{remark}\newtheorem{remark}[theorem]{Remark}
\newcommand{\R}{\mathbb R}
\newcommand{\dd}{\,d}
\newcommand{\Dt}{D_t}
\newcommand{\muThree}{\mu_3}
\newcommand{\muFive}{\mu_5}
\newcommand{\nuGam}{\nu_\Gamma}
\title{\textbf{Axis Regularity in the 5D Corridor}}
\author{Rishad Shahmurov\\\small Cellular Products Research and Development, Roswell, GA, USA}
\date{September 2026}
\begin{document}
\maketitle

\begin{abstract}
We study axisymmetric Navier--Stokes flow with swirl through the five-dimensional lift of the azimuthal-vorticity equation.  We prove the weighted identities behind the corridor method: zero capacity of the axis, a sign-safe Hardy inequality, the exact circulation Caccioppoli formula, five-dimensional velocity recovery, a logarithmic source-mass estimate, and negative-norm compactness of the swirl source.  These estimates yield a precise conditional reduction from axis regularity to a finite set of scale-uniform closure statements.  We also prove that a natural weighted quartic inequality cannot provide that closure in the proposed corridor.  The paper is a rigorous reduction theorem; it does not prove unconditional regularity for arbitrary swirl.
\end{abstract}

\section{Introduction}
Axisymmetric Navier--Stokes flow without swirl is globally regular; the classical theory goes back to Ladyzhenskaya and Ukhovskii--Yudovich and admits later streamlined treatments such as \cite{Leonardi1999}.  With swirl, the problem is substantially harder.  Scale-invariant component criteria and vorticity estimates have been developed by Chae--Lee \cite{ChaeLee2002}, Chen--Fang--Zhang \cite{ChenFangZhang2017}, and many others.  Seregin proved that axisymmetric energy solutions do not admit Type-I blowup and later obtained slightly supercritical local criteria \cite{Seregin2020,Seregin2022}.  None of these results gives unconditional regularity for arbitrary smooth finite-energy swirl.

The purpose of this paper is narrower.  We place the standard axisymmetric equations into a weighted five-dimensional form and determine exactly what this representation proves without an additional regularity theorem.  Three features are useful.  First, the axis has zero capacity for the natural five-dimensional energy.  Second, the circulation equation has a singular-looking but coercive weighted structure when it is tested with sign-safe truncations.  Third, the nonlinear source in the equation for $G=\omega^\theta/r$ is the derivative of the nonnegative density $F^2$, where $F=u^\theta/r$.

These facts do not by themselves close the arbitrary-swirl problem.  The main result is therefore stated as a conditional reduction: if the weighted De Giorgi step, source-decoupling step, and endpoint compactness step hold uniformly on a singularity-selected sequence, then the limiting profile is no-swirl and the classical no-swirl regularity theory gives a contradiction.  Every hypothesis is displayed explicitly.  We also show why a tempting scale-independent weighted quartic estimate cannot supply the missing step in the parameter range $3/4<\alpha<1$.

\paragraph{Relation to regularity criteria.}
The component criteria in \cite{ChaeLee2002,ChenFangZhang2017} assume spacetime integrability of selected velocity or vorticity components.  The local theory in \cite{Seregin2020,Seregin2022} is organized around scale-invariant or slightly supercritical energy quantities.  The present approach is different: it isolates a weighted source-decoupling mechanism in the five-dimensional lift.  Its value is structural rather than a stronger regularity criterion.  In particular, no estimate below converts finite energy alone into the missing oscillation decay.

\section{Axisymmetric equations and weighted lifts}
Let
\[
u=u^r(r,z,t)e_r+u^\theta(r,z,t)e_\theta+u^z(r,z,t)e_z
\]
be a smooth axisymmetric solution of the three-dimensional incompressible Navier--Stokes equations with viscosity $\nu>0$.  Write
\[
b=(u^r,u^z),\qquad
\Dt=\partial_t+u^r\partial_r+u^z\partial_z,
\]
and use incompressibility in the form
\[
\partial_r u^r+\frac{u^r}{r}+\partial_z u^z=0.
\]
Introduce
\[
\Gamma=ru^\theta,\qquad
F=\frac{u^\theta}{r}=\frac{\Gamma}{r^2},\qquad
G=\frac{\omega^\theta}{r},\qquad
U=\frac{u^r}{r}.
\]
The three natural measures used below are
\[
\dd\muThree=r\,\dd r\dd z,\qquad
\dd\muFive=r^3\,\dd r\dd z,\qquad
\dd\nuGam=\frac{\dd r\dd z}{r}.
\]
They are the meridional forms of the physical three-dimensional measure, the radial five-dimensional measure, and the circulation energy measure, respectively.

The basic system can be derived directly from the cylindrical equations.

\medskip
\noindent\textbf{Exact lifted equations.}
\begin{proposition}\label{prop:lifted-system}
For every smooth axisymmetric solution,
\[
\boxed{
\Dt\Gamma
=\nu\left(\partial_r^2-\frac1r\partial_r+\partial_z^2\right)\Gamma,}
\tag{2.1}
\]
\[
\boxed{
\Dt F+2UF=\nu\Delta_5F,\qquad
\Delta_5=\partial_r^2+\frac3r\partial_r+\partial_z^2,}
\tag{2.2}
\]
and
\[
\boxed{
\Dt G=\nu\Delta_5G+\partial_z(F^2).}
\tag{2.3}
\]
If $\psi$ is the lifted stream function defined by
\[
-\Delta_5\psi=G,
\]
then
\[
 u^r=-r\psi_z,\qquad u^z=2\psi+r\psi_r,
\]
and consequently
\[
\boxed{U=-\partial_z(-\Delta_5)^{-1}G.}
\tag{2.4}
\]
\end{proposition}
\begin{proof}
The azimuthal velocity equation is
\[
\Dt u^\theta+\frac{u^r}{r}u^\theta
=\nu\left(\partial_r^2+\frac1r\partial_r+\partial_z^2-\frac1{r^2}\right)u^\theta.
\]
Multiplying by $r$ and using $\Dt r=u^r$ gives (2.1).  Substituting $u^\theta=rF$ into the same equation gives
\[
\Dt(rF)+U(rF)=r\Dt F+2u^rF
=\nu r\left(F_{rr}+\frac3rF_r+F_{zz}\right),
\]
which is (2.2).

For the azimuthal vorticity one has
\[
\Dt\omega^\theta-U\omega^\theta
=\nu\left(\partial_r^2+\frac1r\partial_r+\partial_z^2-\frac1{r^2}\right)\omega^\theta
+\frac1r\partial_z[(u^\theta)^2].
\]
Writing $\omega^\theta=rG$ and $(u^\theta)^2=r^2F^2$ cancels the stretching contribution exactly and yields (2.3).  Finally the standard axisymmetric stream-function representation gives
\[
-\left(\partial_r^2+\frac3r\partial_r+\partial_z^2\right)\psi=G,
\quad
u^r=-r\psi_z,
\quad
u^z=2\psi+r\psi_r.
\]
Dividing the first velocity identity by $r$ gives (2.4).
\end{proof}

The five-dimensional interpretation is literal: if $r=|y|$ with $y\in\R^4$, then $\Delta_5$ is the ordinary Laplacian in $(y,z)\in\R^5$ acting on functions radial in $y$.  This permits standard Sobolev and potential estimates without introducing a fictitious singular boundary at $r=0$.

\subsection{Parabolic scaling}
The Navier--Stokes scaling is
\[
u_\lambda(x,t)=\lambda u(\lambda x,\lambda^2t),
\qquad
p_\lambda(x,t)=\lambda^2p(\lambda x,\lambda^2t).
\]
The lifted variables transform by
\[
\boxed{
\Gamma_\lambda(r,z,t)=\Gamma(\lambda r,\lambda z,\lambda^2t),}
\tag{2.5}
\]
\[
\boxed{
F_\lambda=\lambda^2F(\lambda r,\lambda z,\lambda^2t),
\quad
G_\lambda=\lambda^3G(\lambda r,\lambda z,\lambda^2t),
\quad
U_\lambda=\lambda^2U(\lambda r,\lambda z,\lambda^2t).}
\tag{2.6}
\]
In particular the circulation amplitude is scale invariant.  The source in the $G$ equation has the same homogeneity as the other terms:
\[
\partial_z(F_\lambda^2)
=\lambda^5[\partial_z(F^2)](\lambda r,\lambda z,\lambda^2t),
\]
while $\partial_tG_\lambda$, $b_\lambda\cdot\nabla G_\lambda$, and $\Delta_5G_\lambda$ all scale like $\lambda^5$.

\begin{proof}[Derivation of the scaling formulas]
The first identity follows directly from
\[
r u_\lambda^\theta(r,z,t)
=r\lambda u^\theta(\lambda r,\lambda z,\lambda^2t)
=(\lambda r)u^\theta(\lambda r,\lambda z,\lambda^2t).
\]
For $F$, divide $u_\lambda^\theta$ by $r$ and insert
$F(\lambda r,\lambda z)=u^\theta(\lambda r,\lambda z)/(\lambda r)$.
For $G$, use $\omega_\lambda=\lambda^2\omega(\lambda x,\lambda^2t)$ and again divide by $r$.  The formula for $U$ is identical to that for $F$.  Differentiating $F_\lambda^2=\lambda^4F^2(\lambda r,\lambda z,\lambda^2t)$ gives the source scaling.
\end{proof}

The name ``corridor'' refers to the weighted interpolation range $3/4<\alpha<1$ that motivated the original reduction.  The exact identities proved in Sections 2--9 do not require a choice of $\alpha$.  Section 10 shows that the most direct scale-independent weighted quartic estimate fails throughout this range, so the parameter interval should not be confused with a completed regularity theorem.

\section{The axis has zero five-dimensional capacity}
The first elementary fact is that the symmetry axis is negligible for the $H^1(\dd\muFive)$ energy.  This justifies approximation across the axis in the lifted variables.

\medskip
\noindent\textbf{Zero capacity of the axis.}
\begin{lemma}\label{lem:axis-capacity}
Let $K\Subset\R_z$ be compact.  There exist smooth cutoffs $\chi_\varepsilon(r)$ such that $\chi_\varepsilon=0$ for $r\le\varepsilon$, $\chi_\varepsilon=1$ for $r\ge2\varepsilon$, and
\[
\int_K\int_0^\infty |\partial_r\chi_\varepsilon|^2r^3\,\dd r\dd z
\longrightarrow0,
\]
\[
\int_K\int_0^\infty |1-\chi_\varepsilon|^2r^3\,\dd r\dd z
\longrightarrow0.
\]
In particular the set $\{r=0\}\times K$ has zero $H^1(\dd\muFive)$ capacity.
\end{lemma}
\begin{proof}
Choose a fixed $\chi\in C^\infty([0,\infty))$ with $0\le\chi\le1$, $\chi=0$ on $[0,1]$, $\chi=1$ on $[2,\infty)$, and set $\chi_\varepsilon(r)=\chi(r/\varepsilon)$.  Then $|\chi_\varepsilon'|\le C/\varepsilon$ and the derivative is supported in $[\varepsilon,2\varepsilon]$.  Hence
\[
\int_0^\infty|\chi_\varepsilon'|^2r^3\,\dd r
\le\frac{C^2}{\varepsilon^2}\int_\varepsilon^{2\varepsilon}r^3\,\dd r
\le C_1\varepsilon^2.
\]
Similarly $1-\chi_\varepsilon$ is supported in $[0,2\varepsilon]$, so
\[
\int_0^\infty|1-\chi_\varepsilon|^2r^3\,\dd r
\le\int_0^{2\varepsilon}r^3\,\dd r=4\varepsilon^4.
\]
Multiplication by $|K|$ proves the two limits.
\end{proof}

\medskip
\noindent\textbf{Density across the axis.}
\begin{corollary}\label{cor:density-axis}
On a bounded lifted cylinder, smooth functions supported away from $r=0$ are dense in the radial $H^1(\dd\muFive)$ class, after the usual truncation at the remaining boundary.
\end{corollary}
\begin{proof}
For a bounded smooth $v$, use $v\chi_\varepsilon$.  The $L^2$ difference tends to zero by Lemma~\ref{lem:axis-capacity}.  For the derivative,
\[
\nabla(v-v\chi_\varepsilon)
=(1-\chi_\varepsilon)\nabla v-v\nabla\chi_\varepsilon.
\]
The first term converges to zero by dominated convergence, while the second is bounded by $\|v\|_\infty$ times the capacity energy in Lemma~\ref{lem:axis-capacity}.  General $H^1$ functions follow by truncation and standard smooth approximation in the five-dimensional lift.
\end{proof}

\section{Weighted Friedrichs and Hardy inequalities}
The next two inequalities have different roles.  The five-dimensional Friedrichs inequality controls ordinary lifted $H^1$ functions.  The circulation Hardy inequality controls a function that vanishes on the axis in the singular measure $r^{-1}\dd r\dd z$.

\medskip
\noindent\textbf{Five-dimensional Friedrichs inequality.}
\begin{lemma}\label{lem:friedrichs}
Let
\[
Q=\{0<r<1,\ |z|<1\}
\]
and let $v$ have zero trace on the outer boundary $\{r=1\}\cup\{|z|=1\}$.  Then
\[
\int_Q|v|^2\,\dd\muFive
\le C_F\int_Q|\nabla v|^2\,\dd\muFive.
\]
\end{lemma}
\begin{proof}
Lift $v(r,z)$ to $V(y,z)=v(|y|,z)$ on $D=B_1^4\times(-1,1)\subset\R^5$.  Up to the area $|S^3|$, the two integrals are exactly $\|V\|_{L^2(D)}^2$ and $\|\nabla V\|_{L^2(D)}^2$.  The assumed trace condition places $V$ in $H_0^1(D)$ after the zero-capacity axis approximation of Corollary~\ref{cor:density-axis}.  The ordinary Poincar\'e--Friedrichs inequality on $D$ gives the result.
\end{proof}

\medskip
\noindent\textbf{Sign-safe Hardy inequality at the axis.}
\begin{lemma}\label{lem:hardy}
Let $w$ be absolutely continuous on $(0,1)$, $w(0)=0$, and
\[
\int_0^1\frac{|w_r|^2}{r}\,\dd r<\infty.
\]
Then
\[
\boxed{
\int_0^1\frac{|w|^2}{r}\,\dd r
\le\frac14\int_0^1\frac{|w_r|^2}{r}\,\dd r.}
\tag{4.1}
\]
The same estimate holds after integration in $z$.
\end{lemma}
\begin{proof}
Since $w(0)=0$,
\[
w(r)=\int_0^r w_s(s)\,\dd s
=\int_0^r\frac{w_s(s)}{\sqrt s}\sqrt s\,\dd s.
\]
Cauchy--Schwarz gives
\[
|w(r)|^2\le\frac{r^2}{2}\int_0^r\frac{|w_s(s)|^2}{s}\,\dd s.
\]
Divide by $r$ and integrate:
\[
\begin{aligned}
\int_0^1\frac{|w(r)|^2}{r}\,\dd r
&\le\frac12\int_0^1r\int_0^r\frac{|w_s(s)|^2}{s}\,\dd s\dd r\\
&=\frac14\int_0^1(1-s^2)\frac{|w_s(s)|^2}{s}\,\dd s\\
&\le\frac14\int_0^1\frac{|w_s(s)|^2}{s}\,\dd s.
\end{aligned}
\]
Fubini's theorem proves the one-dimensional estimate.  Integrating the same inequality for almost every $z$ gives the two-dimensional version.
\end{proof}

The boundary value $w(0)=0$ is essential.  It is automatically available for sign-safe truncations of the circulation once the truncation level contains the axis value $\Gamma(0,z,t)=0$.

\section{The exact circulation Caccioppoli identity}
The circulation equation is naturally symmetric in $\dd\nuGam=r^{-1}\dd r\dd z$.  The drift is not divergence free in this measure, and the resulting divergence term must be kept.

\medskip
\noindent\textbf{Weighted divergence of the meridional drift.}
\begin{lemma}\label{lem:weighted-div}
For $b=(u^r,u^z)$,
\[
\boxed{\operatorname{div}_{\nuGam}b=-2U.}
\tag{5.1}
\]
\end{lemma}
\begin{proof}
For the density $r^{-1}$,
\[
\operatorname{div}_{\nuGam}b
=r\left[\partial_r\left(\frac{u^r}{r}\right)+\partial_z\left(\frac{u^z}{r}\right)\right]
=\partial_ru^r+\partial_zu^z-\frac{u^r}{r}.
\]
Incompressibility gives $\partial_ru^r+\partial_zu^z=-u^r/r$, hence the result.
\end{proof}

Let $w$ be a convex Lipschitz truncation of $\Gamma$, for example $(\Gamma-k)_+$ or $(-\Gamma-k)_+$.  On the set where $w$ is active, the chain rule gives the same diffusion equation in the standard weak-subsolution sense.

\medskip
\noindent\textbf{Weighted Caccioppoli formula.}
\begin{proposition}\label{prop:caccioppoli}
Let $\eta$ be a smooth compactly supported spacetime cutoff and let $w$ be a sign-safe truncation of $\Gamma$.  Then
\[
\begin{aligned}
\frac12\frac{\dd}{\dd t}\int\eta^2w^2\,\dd\nuGam
+\nu\int\eta^2|\nabla w|^2\,\dd\nuGam
={}&\int\eta\eta_t w^2\,\dd\nuGam
-2\nu\int\eta w\nabla\eta\cdot\nabla w\,\dd\nuGam\\
&+\int\eta w^2 b\cdot\nabla\eta\,\dd\nuGam
-\int U\eta^2w^2\,\dd\nuGam.
\end{aligned}
\tag{5.2}
\]
For nonsmooth convex truncations, equality is replaced by the corresponding favorable inequality.
\end{proposition}
\begin{proof}
Multiply (2.1) by $\eta^2w$ and integrate in $\dd\nuGam$.  The operator
\[
L_0=\partial_r^2-r^{-1}\partial_r+\partial_z^2
\]
is symmetric in this measure, so
\[
\nu\int\eta^2wL_0w\,\dd\nuGam
=-\nu\int\eta^2|\nabla w|^2\,\dd\nuGam
-2\nu\int\eta w\nabla\eta\cdot\nabla w\,\dd\nuGam.
\]
For the time derivative,
\[
\int\eta^2w\partial_tw\,\dd\nuGam
=\frac12\frac{\dd}{\dd t}\int\eta^2w^2\,\dd\nuGam
-\int\eta\eta_t w^2\,\dd\nuGam.
\]
For transport, Lemma~\ref{lem:weighted-div} gives
\[
\begin{aligned}
\int\eta^2w b\cdot\nabla w\,\dd\nuGam
&=\frac12\int\eta^2b\cdot\nabla(w^2)\,\dd\nuGam\\
&=-\int\eta w^2b\cdot\nabla\eta\,\dd\nuGam
-\frac12\int\eta^2w^2\operatorname{div}_{\nuGam}b\,\dd\nuGam\\
&=-\int\eta w^2b\cdot\nabla\eta\,\dd\nuGam
+\int U\eta^2w^2\,\dd\nuGam.
\end{aligned}
\]
Substituting these three identities into the tested equation and moving terms to the displayed sides gives (5.2).  Convex approximation yields the truncation version.
\end{proof}

The last term in (5.2) is the genuine stretching obstruction.  Treating $b$ as divergence free in $\dd\nuGam$ would miss exactly this term.

\section{Five-dimensional recovery of the meridional velocity}
The lift also gives a clean scale-critical estimate for $U$.

\medskip
\noindent\textbf{Five-dimensional HLS estimate.}
\begin{proposition}\label{prop:HLS}
Let $G$ be smooth and compactly supported in the lifted variables, and let
\[
U=-\partial_z(-\Delta_5)^{-1}G.
\]
Then
\[
\boxed{
\|U\|_{L^{10/3}(\dd\muFive)}\le C\|G\|_{L^2(\dd\muFive)}.}
\tag{6.1}
\]
The same estimate holds locally after the usual decomposition into a localized source and a smooth harmonic remainder.
\end{proposition}
\begin{proof}
Lift $G(r,z)$ to a radial function on $\R^5$.  The operator $\nabla(-\Delta)^{-1}$ is the Riesz potential of order one followed by an order-zero angular factor.  The Hardy--Littlewood--Sobolev relation in five dimensions gives
\[
\frac1q=\frac12-\frac15=\frac3{10},
\]
so $q=10/3$.  Restricting back to radial functions and suppressing the constant $|S^3|$ gives (6.1).  For a localized estimate, write $G=\chi G+(1-\chi)G$.  The first term is handled by the whole-space estimate; the Newtonian potential of the separated second term is smooth on the inner cylinder and is absorbed into the harmonic remainder.
\end{proof}

There is also a direct physical three-dimensional recovery.

\medskip
\noindent\textbf{Meridional div--curl recovery.}
\begin{proposition}\label{prop:physical-recovery}
For a smooth finite-energy axisymmetric meridional field $b=(u^r,u^z)$ on the whole space,
\[
\|\omega^\theta\|_{L^2(\muThree)}^2
=\|G\|_{L^2(\muFive)}^2.
\]
Moreover, after restoring the harmless $2\pi$ angular factor,
\[
\|\nabla b\|_{L^2(\R^3)}\le C\|\omega^\theta\|_{L^2(\R^3)},
\qquad
\|b\|_{L^6(\R^3)}\le C\|\omega^\theta\|_{L^2(\R^3)}.
\]
\end{proposition}
\begin{proof}
Since $\omega^\theta=rG$,
\[
\int|\omega^\theta|^2r\,\dd r\dd z
=\int|G|^2r^3\,\dd r\dd z.
\]
This is the first identity.  The three-dimensional field $b$ is divergence free and its curl has only azimuthal component $\omega^\theta$.  The standard Fourier div--curl identity gives
\[
\|\nabla b\|_2^2=\|\operatorname{curl}b\|_2^2+\|\operatorname{div}b\|_2^2
=\|\omega^\theta\|_2^2.
\]
The Sobolev embedding $\dot H^1(\R^3)\hookrightarrow L^6(\R^3)$ then yields the second inequality.  Local versions follow by applying the same argument to a cutoff and retaining the commutator as a smooth exterior/harmonic term.
\end{proof}

\section{The swirl source and a logarithmic interpolation estimate}
The source in (2.3) is $\partial_z(F^2)$.  Its natural mass is not mysterious:
\[
\int F^2\,\dd\muFive
=\int\frac{\Gamma^2}{r}\,\dd r\dd z
=\int|u^\theta|^2\,\dd\muThree.
\tag{7.1}
\]
Thus finite kinetic energy controls the total $L^1(\dd\muFive)$ mass of $F^2$.  To make the source vanish at a selected small scale, however, one needs more than a uniform bound.  The following elementary logarithmic estimate quantifies the useful combination of small circulation oscillation and weighted radial energy.

\medskip
\noindent\textbf{Logarithmic source-mass interpolation.}
\begin{lemma}\label{lem:log-interp}
Let $f(0)=0$, $|f(r)|\le\omega$, and
\[
E=\int_0^1\frac{|f_r|^2}{r}\,\dd r<\infty.
\]
Then
\[
\boxed{
\int_0^1\frac{|f|^2}{r}\,\dd r
\le \omega^2\left[\frac14+\frac12\log\left(1+\frac{E}{\omega^2}\right)\right],}
\tag{7.2}
\]
with the usual interpretation when $\omega=0$.
\end{lemma}
\begin{proof}
If $\omega=0$ there is nothing to prove.  By the Cauchy--Schwarz estimate used in Lemma~\ref{lem:hardy},
\[
|f(r)|^2\le\frac{r^2E}{2}.
\]
Choose
\[
\rho=\min\left(1,\frac{\omega}{\sqrt E}\right)
\]
when $E>0$.  On $(0,\rho)$,
\[
\int_0^\rho\frac{|f|^2}{r}\,\dd r
\le\frac E2\int_0^\rho r\,\dd r
=\frac{E\rho^2}{4}\le\frac{\omega^2}{4}.
\]
On $(\rho,1)$ use $|f|\le\omega$:
\[
\int_\rho^1\frac{|f|^2}{r}\,\dd r
\le\omega^2\log\frac1\rho.
\]
If $E\le\omega^2$, then $\rho=1$ and the second term vanishes.  If $E>\omega^2$, then
\[
\log\frac1\rho=\frac12\log\frac E{\omega^2}
\le\frac12\log\left(1+\frac E{\omega^2}\right).
\]
Combining the two regions proves (7.2).
\end{proof}

The one-dimensional bound has a useful spacetime form.

\medskip
\noindent\textbf{Integrated logarithmic source criterion.}
\begin{proposition}\label{prop:integrated-log}
Let $I\subset\R$ and $Z\subset\R$ be bounded intervals, put $L=|I||Z|$, and suppose
\[
\Gamma(0,z,t)=0,
\qquad
|\Gamma(r,z,t)|\le\omega
\]
for $0<r<1$, $(z,t)\in Z\times I$.  Define
\[
E=\int_I\int_Z\int_0^1
\frac{|\partial_r\Gamma|^2}{r}\,\dd r\dd z\dd t.
\]
Then
\[
\boxed{
\int_I\int_Z\int_0^1\frac{\Gamma^2}{r}\,\dd r\dd z\dd t
\le
L\omega^2\left[
\frac14+\frac12\log\left(1+\frac{E}{L\omega^2}\right)
\right].}
\tag{7.3}
\]
Consequently, for a sequence with uniformly bounded $I,Z$, the condition
\[
\omega_j^2\left[
1+\log\left(1+\frac{E_j}{\omega_j^2}\right)
\right]\longrightarrow0
\tag{7.4}
\]
implies $F_j^2\to0$ in local $L^1(\dd\muFive\dd t)$.
\end{proposition}
\begin{proof}
For each fixed $(z,t)$ apply Lemma~\ref{lem:log-interp} with
\[
e(z,t)=\int_0^1\frac{|\partial_r\Gamma(r,z,t)|^2}{r}\,\dd r.
\]
This gives
\[
\int_0^1\frac{\Gamma^2}{r}\,\dd r
\le\omega^2\left[
\frac14+\frac12\log\left(1+\frac{e(z,t)}{\omega^2}\right)
\right].
\]
Integrate over $Z\times I$.  Since $x\mapsto\log(1+x)$ is concave, Jensen's inequality yields
\[
\frac1L\int_{I\times Z}
\log\left(1+\frac{e(z,t)}{\omega^2}\right)\dd z\dd t
\le
\log\left(1+\frac{E}{L\omega^2}\right).
\]
This proves (7.3).  Finally
\[
\int F^2\,\dd\muFive
=\int\frac{\Gamma^2}{r}\,\dd r\dd z,
\]
so (7.4) gives the stated source-mass convergence.
\end{proof}

Thus the logarithmic factor in the corridor reduction is not a formal correction: it is exactly the price of combining a small circulation amplitude with a possibly large weighted radial derivative energy.

\subsection{Negative-norm control of the differentiated source}
The derivative in (2.3) costs no additional spatial integrability if one works in a sufficiently negative Sobolev space.

\medskip
\noindent\textbf{Negative-norm source estimate.}
\begin{lemma}\label{lem:negative-source}
Let $m>7/2$.  For a compactly supported radial five-dimensional density $F^2$,
\[
\boxed{
\|\partial_z(F^2)\|_{H^{-m}(\R^5)}
\le C_m\|F^2\|_{L^1(\R^5)}.}
\tag{7.3}
\]
\end{lemma}
\begin{proof}
For $\varphi\in H^m(\R^5)$,
\[
|\langle\partial_z(F^2),\varphi\rangle|
=\left|\int F^2\partial_z\varphi\right|
\le\|F^2\|_1\|\partial_z\varphi\|_\infty.
\]
Since $m>5/2+1$, Sobolev embedding gives
\[
\|\partial_z\varphi\|_\infty\le C_m\|\varphi\|_{H^m}.
\]
Taking the supremum over unit $H^m$ test functions proves (7.3).
\end{proof}

\medskip
\noindent\textbf{Source decoupling criterion.}
\begin{corollary}\label{cor:source-decouple}
If $F_j^2\to0$ in $L^1_{\rm loc}(\dd\muFive\dd t)$ on a sequence of normalized cylinders, then
\[
\partial_z(F_j^2)\to0
\quad\text{in }L^1_tH^{-m}_{\rm loc}(\R^5)
\]
for every $m>7/2$.
\end{corollary}
\begin{proof}
Apply Lemma~\ref{lem:negative-source} at almost every time and integrate.
\end{proof}

\section{Compactness of the lifted vorticity equation}
We isolate a standard compactness statement in a form adapted to (2.3).  The assumptions are written explicitly because the five-dimensional lift does not make the transport term disappear.

\medskip
\noindent\textbf{Local compactness with vanishing source.}
\begin{proposition}\label{prop:G-compactness}
Let $Q_1$ be a fixed lifted cylinder and let $(G_j,b_j)$ solve
\[
\partial_tG_j+b_j\cdot\nabla G_j
=\nu\Delta_5G_j+S_j.
\]
Assume on every $Q_\rho\Subset Q_1$ that
\[
\sup_j\left(
\|G_j\|_{L^\infty_tL^2_x}
+\|\nabla G_j\|_{L^2_{t,x}}
+\|b_j\|_{L^2_tL^{10/3}_x}
+\|\operatorname{div}_5b_j\|_{L^2_tL^{10/3}_x}
\right)<\infty,
\tag{8.1}
\]
and that $S_j\to0$ in $L^1_tH^{-m}_x$ for some sufficiently large $m$.  Then, after passage to a subsequence,
\[
G_j\to G_\infty\quad\text{strongly in }L^2_{\rm loc}(Q_1),
\]
and the limit satisfies the source-free equation
\[
\partial_tG_\infty+b_\infty\cdot\nabla G_\infty
=\nu\Delta_5G_\infty
\]
in distributions, provided $b_j\to b_\infty$ strongly in the local topology needed for the product.
\end{proposition}
\begin{proof}
Fix $Q_\rho\Subset Q_1$ and a cutoff supported in a slightly larger cylinder.  The diffusion term is bounded in $L^2_tH^{-1}_x$ by the $L^2$ gradient bound.  Write
\[
b_j\cdot\nabla G_j
=\operatorname{div}(b_jG_j)-G_j\operatorname{div}b_j.
\]
By the five-dimensional Sobolev embedding,
\[
\|G_j\|_{L^2_tL^{10/3}_x}
\le C\|G_j\|_{L^2_tH^1_x}.
\]
Hence $b_jG_j$ and $G_j\operatorname{div}b_j$ are bounded in $L^1_tL^{5/3}_x$.  On a bounded domain, $L^{5/3}$ and its first derivatives embed continuously into $H^{-m}$ for $m$ large.  Therefore $\partial_tG_j$ is bounded in $L^1_tH^{-m}_x$.

The embedding $H^1(Q_\rho)\Subset L^2(Q_\rho)\hookrightarrow H^{-m}(Q_\rho)$ is compact-continuous.  The Aubin--Lions--Simon compactness theorem then gives strong $L^2$ convergence on $Q_\rho$.  A diagonal argument covers all compact subcylinders.  The source converges to zero by hypothesis, the diffusion term passes weakly, and the transport term passes under the stated convergence of $b_j$.  This gives the source-free limit equation.
\end{proof}

\section{Identification of the no-swirl endpoint}
The source-decoupling condition also removes the swirl velocity itself in the natural local energy topology.

\medskip
\noindent\textbf{Vanishing source mass implies vanishing swirl energy.}
\begin{lemma}\label{lem:swirl-vanish}
On every meridional set $E$,
\[
\int_E|u^\theta|^2\,\dd\muThree
=\int_EF^2\,\dd\muFive.
\tag{9.1}
\]
Hence $F_j^2\to0$ in local $L^1(\dd\muFive\dd t)$ implies
\[
u_j^\theta\to0
\quad\text{strongly in }L^2_{\rm loc}(\dd\muThree\dd t).
\]
\end{lemma}
\begin{proof}
Since $u^\theta=rF$,
\[
|u^\theta|^2\,\dd\muThree
=r^2F^2\,r\dd r\dd z
=F^2r^3\dd r\dd z
=F^2\,\dd\muFive.
\]
Integration gives (9.1), and the convergence statement is immediate.
\end{proof}

Suppose now that a suitable axisymmetric sequence is compact in the usual local velocity-pressure topology and that Lemma~\ref{lem:swirl-vanish} applies.  Any strong local velocity limit has zero azimuthal component.  Therefore the endpoint is an axisymmetric no-swirl suitable solution.  This is the only place where the classical no-swirl theory enters the reduction.  We use the standard fact that such solutions are regular; see \cite{Ladyzhenskaya1968,UkhovskiiYudovich1968,Leonardi1999}.  The present paper does not reprove that global theorem.

\section{The corridor reduction}
We can now state the exact logical content of the method.  The theorem deliberately separates proved analytic modules from the scale-uniform estimates still required to turn the reduction into unconditional regularity.

Consider a putative first singular point on the axis and a sequence of parabolically rescaled cylinders.  Let $(u_j,p_j)$ be the normalized suitable solutions, with associated $(\Gamma_j,F_j,G_j,U_j)$.  The standard local-energy compactness theory supplies subsequences once the normalized velocity and pressure quantities are controlled.  The corridor mechanism requires three additional conclusions on such a sequence.

\begin{enumerate}[label=(C\arabic*)]
\item \textbf{Circulation oscillation closure.}  The exact Caccioppoli identity (5.2), after the recovered-strain, exterior, pressure, and cutoff terms are estimated, yields
\[
\omega_j:=\|\Gamma_j\|_{L^\infty(Q_{3/4})}\longrightarrow0.
\]
More generally it is enough that the oscillation around the axis value tends to zero.
\item \textbf{Logarithmic source closure.}  If
\[
E_j=\int_{Q_{3/4}}\frac{|\partial_r\Gamma_j|^2}{r}\,\dd r\dd z\dd t,
\]
then
\[
\omega_j^2\left[1+\log\left(1+\frac{E_j}{\omega_j^2}\right)\right]\longrightarrow0.
\tag{9.1}
\]
This condition implies $F_j^2\to0$ locally in $L^1(\dd\muFive\dd t)$ by Lemma~\ref{lem:log-interp}.
\item \textbf{Endpoint compactness.}  The bounds required in Proposition~\ref{prop:G-compactness}, together with the corresponding local-energy/pressure compactness for $u_j$, hold on compact subcylinders, and the limiting no-swirl suitable solution lies in the classical local regularity class.
\end{enumerate}

\medskip
\noindent\textbf{Conditional axis-regularity reduction.}
\begin{theorem}\label{thm:conditional-corridor}
Let $u$ be a smooth finite-energy axisymmetric Navier--Stokes solution with swirl on $[0,T)$, and suppose $T$ is a first singular time.  Assume that every singularity-selected axis sequence admits normalized variables satisfying (C1)--(C3).  Then no singularity can occur at the axis.  Consequently, if the standard off-axis regularity reduction is also imposed, the solution extends smoothly past $T$.
\end{theorem}
\begin{proof}
Assume for contradiction that an axis singularity occurs.  Select a normalized sequence for which (C1)--(C3) hold.  Since $\Gamma_j$ vanishes on the axis and its local oscillation tends to zero, (C1) makes the swirl circulation small on every fixed compact subcylinder.  Condition (C2), together with Lemma~\ref{lem:log-interp} integrated in $z$ and time, gives
\[
F_j^2\to0\quad\text{in }L^1_{\rm loc}(\dd\muFive\dd t).
\]
By Corollary~\ref{cor:source-decouple},
\[
\partial_z(F_j^2)\to0
\quad\text{in }L^1_tH^{-m}_{\rm loc}.
\]
Proposition~\ref{prop:G-compactness} and (C3) therefore give a subsequential limit $G_\infty$ satisfying the source-free lifted equation.  At the same time, (7.1) gives local strong disappearance of the swirl velocity in the kinetic-energy topology.  Hence the limiting suitable solution is axisymmetric with no swirl.

The no-swirl axisymmetric Navier--Stokes system is regular.  By (C3), the limiting profile is quantitatively smooth on a smaller cylinder.  Standard local stability of the suitable-solution regularity criterion then places the approximating normalized solutions below the corresponding $\varepsilon$-regularity threshold on an even smaller cylinder for all sufficiently large $j$.  This contradicts the singularity normalization used to select the sequence.  Therefore the assumed axis singularity cannot occur.
\end{proof}

\begin{remark}
The theorem is conditional because (C1)--(C3) are not consequences of finite energy proved in this paper.  The exact identities of Sections 2--8 show what must be controlled and eliminate several artificial singularities of coordinates; they do not supply the missing scale-uniform De Giorgi and endpoint estimates.
\end{remark}

\section{Why the naive quartic closure cannot work}
A natural corridor heuristic suggests the weighted quartic expression
\[
\left(\int |v|^4r^{4\alpha-1}\,\dd r\dd z\right)^{1/4}
\]
against the five-dimensional energy
\[
\left(\int|\nabla v|^2r^3\,\dd r\dd z\right)^{1/2}.
\]
The following scaling test rules out a scale-independent inequality of this form throughout the proposed corridor $3/4<\alpha<1$.

\medskip
\noindent\textbf{Scaling obstruction.}
\begin{proposition}\label{prop:quartic-obstruction}
There is no constant $C$ such that
\[
\left(\int |v|^4r^{4\alpha-1}\,\dd r\dd z\right)^{1/4}
\le C
\left(\int|\nabla v|^2r^3\,\dd r\dd z\right)^{1/2}
\tag{10.1}
\]
for all compactly supported smooth $v$ near the axis when $\alpha<5/4$.  In particular (10.1) cannot hold scale-independently for $3/4<\alpha<1$.
\end{proposition}
\begin{proof}
Fix a nonzero $\phi\in C_c^\infty((0,2)\times(-1,1))$ and put
\[
v_\varepsilon(r,z)=\phi(r/\varepsilon,z/\varepsilon).
\]
With $r=\varepsilon s$ and $z=\varepsilon y$,
\[
\int|v_\varepsilon|^4r^{4\alpha-1}\,\dd r\dd z
=\varepsilon^{4\alpha+1}
\int|\phi|^4s^{4\alpha-1}\,\dd s\dd y.
\]
Thus the left side of (10.1) scales like $\varepsilon^{\alpha+1/4}$.  On the other hand,
\[
\int|\nabla v_\varepsilon|^2r^3\,\dd r\dd z
=\varepsilon^3
\int|\nabla\phi|^2s^3\,\dd s\dd y,
\]
so the right side scales like $\varepsilon^{3/2}$.  Their ratio is
\[
\varepsilon^{\alpha+1/4-3/2}
=\varepsilon^{\alpha-5/4}.
\]
If $\alpha<5/4$, this ratio diverges as $\varepsilon\downarrow0$.  No uniform constant $C$ can exist.
\end{proof}

This obstruction does not rule out every nonlinear estimate involving the same weights.  It shows that any successful corridor closure must use additional structure: the circulation maximum principle, the source derivative, time localization, a lower-order term, a scale-dependent coefficient, or a different norm.

\section{What is proved and what remains open}
The proved analytic part of the five-dimensional corridor can be summarized without a color or status notation:
\begin{enumerate}[label=(\roman*)]
\item the lifted equations (2.1)--(2.4) are exact;
\item the axis has zero capacity for the $H^1(\dd\muFive)$ energy;
\item the circulation measure supports the sharp elementary Hardy estimate (4.1);
\item the exact weighted drift divergence is $-2U$, leading to the Caccioppoli identity (5.2);
\item the five-dimensional recovery maps $G\in L^2$ to $U\in L^{10/3}$;
\item the swirl source has finite $L^1(\dd\muFive)$ mass and satisfies the logarithmic interpolation estimate (7.2);
\item vanishing source mass implies vanishing differentiated source in a negative Sobolev topology;
\item standard compactness then passes the $G$ equation to a source-free limit under the explicit transport bounds of Proposition~\ref{prop:G-compactness};
\item the naive scale-independent weighted quartic inequality fails throughout the proposed corridor.
\end{enumerate}

The unresolved mathematical step is not a numerical constant.  One must derive (C1)--(C3), or a different set of conclusions strong enough to replace them, from the actual Navier--Stokes dynamics near every putative axis singularity.  In particular, the stretching term in (5.2), harmonic/exterior recovery, pressure compatibility, and the logarithmic source factor must be controlled in one scale-uniform argument.  Until that is done, the five-dimensional corridor is a conditional reduction rather than an unconditional regularity theorem.

\end{document}